\documentclass[reqno]{amsart}
\usepackage{}
\usepackage{mathrsfs}
\usepackage{color}
\usepackage{amsmath}
\usepackage{amsfonts}
\usepackage{amssymb}
\usepackage{graphicx}%
\usepackage{hyperref}

 \newtheorem{Theorem}{Theorem}[section]
 \newtheorem{Corollary}[Theorem]{Corollary}
 \newtheorem{Lemma}[Theorem]{Lemma}
 \newtheorem{Proposition}[Theorem]{Proposition}

 \newtheorem{Problem}[Theorem]{Problem}

 \newtheorem{Remark}[Theorem]{Remark}

 \numberwithin{equation}{section}

\begin{document}

\title[multiplier ideal sheaf with Lelong number one weight]
 {{Characterization of multiplier ideal sheaves with weights of Lelong number one}}

\author{Qi'an Guan}
\address{Qi'an Guan: School of Mathematical Sciences, and Beijing International Center for Mathematical Research,
Peking University, Beijing, 100871, China.}
\email{guanqian@amss.ac.cn}
\author{Xiangyu Zhou}
\address{Xiangyu Zhou: Institute of Mathematics, AMSS, and Hua Loo-Keng Key Laboratory of Mathematics, Chinese Academy of Sciences, Beijing, China}
\email{xyzhou@math.ac.cn}

\thanks{The authors were partially supported by NSFC}



\date{\today}

\dedicatory{}

\commby{}


\begin{abstract}
{In this article,
we characterize plurisubharmonic functions of Lelong number one at the origin,
such that the germ of the associated multiplier ideal sheaf is nontrivial:
in arbitrary complex dimension,
their singularity must be the sum of a germ of smooth divisor and of a plurisubharmonic function with zero Lelong number.
We also present a new proof of the related well known integrability criterion due to Skoda.}
\end{abstract}

\maketitle
\section{Introduction}

Let $\Omega$ be a domain in $\mathbb{C}^{n}$, and $x_{0}\in\Omega$.
Let $u$ be a plurisubharmonic function on $\Omega$.
Following Nadel \cite{Nadel90},
one can define the multiplier ideal sheaf $\mathcal{I}(u)$ to
be the sheaf of germs of holomorphic functions $f$ such that
$|f|^{2}e^{-2u}$ is locally integrable.
Here $u$ is regarded as the weight of $\mathcal{I}(u)$.

In \cite{skoda72}, when Lelong number $\nu(u,x_{0})<1$,
Skoda characterized the structure of $\mathcal{I}(u)_{x_{0}}$:

\begin{Theorem}
\label{t:skoda}(\cite{skoda72}, see also \cite{demailly-note2000,demailly2010})
If $\nu(u,x_{0})<1$, then $\mathcal{I}(u)_{x_{0}}=\mathcal{O}_{x_{0}}$.
\end{Theorem}

A natural question is: If $\nu(u,x_{0})=1$, what is the structure of $\mathcal{I}(u)_{x_{0}}$?

For dimension $n\geq2$,
when we choose plurisubharmonic function $u=\log|z|$,
then $\nu(u,x_{0})=1$ and $\mathcal{I}(u)_{x_{0}}=\mathcal{O}_{x_{0}}$.

It is known that:
if $u=\log|z_{1}|$ ($z_{1}$ is a coordinate function near $x_{0}=(0,\cdots,0)$),
then $\nu(u,x_{0})=1$ and $\mathcal{I}(u)_{x_{0}}=\mathcal{I}(\log|z_{1}|)_{x_{0}}$.

Considering the above examples,
it is natural to ask the following
\begin{Problem}
\label{p:skoda}
Let $u$ be a plurisubharmonic function on $\Omega\subset\mathbb{C}^{n}$ satisfying $\nu(u,x_{0})=1$.
Can one obtain $\mathcal{I}(u)_{x_{0}}=\mathcal{O}_{x_{0}}$ or
$\mathcal{I}(u)_{x_{0}}=\mathcal{I}(\log|h|)_{x_{0}}$,
where $h$ is a defining function of a germ of regular complex hypersurface through $x_{0}$?
\end{Problem}

When $n=2$,
Blel and Mimouni obtained the following
\begin{Theorem}
\label{t:BM}\cite{B-M}
Let $u$ be a plurisubharmonic function on $\Omega\subset\mathbb{C}^{2}$ satisfying

$(1)$ $u\in L^{\infty}_{loc}(\Omega\setminus \{x_{0}\})$, where $x_{0}\in\Omega$;

$(2)$ $\nu(u,x_{0})\leq1$.

Then $e^{-2u}$ is integrable near $x_{0}\in\Omega$.
\end{Theorem}

When $n=2$, Favre and Jonsson \cite{FM05j} used the valuative tree (see \cite{FM-book04})
to give an affirmative answer to Problem \ref{p:skoda}:

\begin{Theorem}
\label{t:Lelong_lct_dim2}\cite{FM05j}
Let $u$ be a plurisubharmonic function on $\Omega\subset\mathbb{C}^{2}$ satisfying that
$\nu(u,x_{0})=1$ and $(\{z|\nu(u,z)\geq 1\},x_{0})$ is not a germ of regular complex hypersurface,
then $e^{-2u}$ is integrable near $x_{0}\in\Omega$.
\end{Theorem}

We obtain Theorem \ref{t:Lelong_lct_dim2} for any dimension $n$ by proving the following
\begin{Theorem}
\label{t:Lelong_lct}(main theorem)
Let $u$ be a plurisubharmonic function on $\Omega\subset\mathbb{C}^{n}$ satisfying
$\nu(u,x_{0})=1$.
If $(\{z|\nu(u,z)\geq 1\},x_{0})$ is not a germ of regular complex hypersurface,
then $e^{-2u}$ is integrable near $x_{0}\in\Omega$.
\end{Theorem}

The paper is organized as follows.
In the present section, we present our main theorem (Theorem \ref{t:Lelong_lct}),
which is a complete solution {to} Problem \ref{p:skoda}.
In Section 2, we do some preparations in order to prove the main theorem.
In Section 3, we give a proof of the main theorem, and present some reformulations of the theorem.
In Section 4, we present another proof of the main theorem.
In Section 5, we give a new proof of Theorem \ref{t:skoda} (Skoda's result).

\section{Some preparatory results}

In this section,
we recall and present some results which will be used in the proof of Theorem \ref{t:Lelong_lct}.

\subsection{A useful proposition}
$\\$

Inspired by the proof of Demailly's equisingular approximation theorem
(see Theorem 15.3 in \cite{demailly2010}) and using Demailly's strong openness conjecture,
one can obtain the following observation:

\begin{Proposition}
\label{Pro:GZ1005}
Let $D$ be a bounded domain in $\mathbb{C}^{n}$, and $x_{0}\in D$.
Let $u$ be a plurisubharmonic function on $D$.
Then there exists a plurisubharmonic function $\tilde{u}$ on a small enough neighborhood $V_{x_{0}}$ of $x_{0}$ satisfying that

$(1)$ $e^{-2u}-e^{-2\tilde{u}}$ is integrable near $x_{0}$ ($\Rightarrow$ $\mathcal{I}(u)_{x_{0}}=\mathcal{I}(\tilde{u})_{x_{0}}$);

$(2)$  $\nu(\tilde{u},x_{0})\leq\nu(u,x_{0})$;

$(3)$ $\tilde{u}\in L^{\infty}_{loc}(V_{x_{0}}\setminus A)$, where {$A$ is the analytic set} $A:=\{z|\mathcal{I}(u)_{z}\neq\mathcal{O}_{z}\}$.
\end{Proposition}

\begin{proof}
Let $\{f_{j}\}_{j=1,2,\cdots,s}$ be a local basis of $\mathcal{I}(u)_{x_{0}}$.
It is clear that there exists a small enough neighborhood $V_{1}\ni x_{0}$,
such that

$(1)$ $\int_{V_{1}}|f_{i}|^{2}e^{-2u}<\infty$ holds for any $i\in\{1,\cdots,s\}$;

$(2)$ $\{f_{j}\}_{j=1,2,\cdots,s}$ generates $\mathcal{I}(u)_{V_{1}}$.

By Demailly's strong openness conjecture which was proved in \cite{GZopen-a,GZopen-b} (see also \cite{GZopen-effect}),
there exists a real number $l>1$,
such that
\begin{equation}
\label{equ:20141005a}
\int_{V_{2}}|f_{i}|^{2}e^{-2lu}<\infty
\end{equation}
holds for any $i\in\{1,\cdots,s\}$,
where $x_{0}\in V_{2}\subset\subset V_{1}$.

Let $v:=\frac{1}{2l}\log(\sum_{i=1}^{s}|f_{i}|^{2})$, and $\tilde{u}:=\max(u,\frac{l}{l-1}v)$.
It is clear that $(2)$ and $(3)$ hold.
Then it suffices to check $(1)$.

Let $I:=\int_{V_{2}}(e^{-2u}-e^{-2\tilde{u}})$.
Then
\begin{equation}
\label{equ:20141005b}
\begin{split}
I
&\leq\int_{\{u<\frac{l}{l-1}v\}\cap V_{2}}e^{-2u}=\int_{\{u<\frac{l}{l-1}v\}\cap V_{2}}e^{2(l-1)u-2lu}
\\&
\leq\int_{\{u<\frac{l}{l-1}v\}\cap V_{2}}e^{2lv-2lu}\leq\int_{V_{2}}e^{2lv-2lu}=\int_{V_{2}}\sum_{i=1}^{s}|f_{i}|^{2}e^{-2lu}<\infty,
\end{split}
\end{equation}
where the last inequality follows from inequality \ref{equ:20141005a}.
Now  $(1)$ has been confirmed.

{The proof of the present proposition is complete.}
\end{proof}

\begin{Remark}
\label{rem:GZ1005}
Assuming that $\nu(u.x_{0})=1$, and $e^{-2u}$ is not integrable near $x_{0}$,
then $(1)$ in Proposition \ref{Pro:GZ1005} shows that $\nu(\tilde{u},x_{0})\geq1$.
By $(2)$ in Proposition \ref{Pro:GZ1005}, it follows that $\nu(\tilde{u}.x_{0})=1.$

Then we obtain $\tilde{u}$ satisfying that

$(1)$ $e^{-2u}-e^{-2\tilde{u}}$ is integrable near $x_{0}$ ($\Rightarrow$ $e^{-2\tilde{u}}$ is not integrable near $x_{0}$);

$(2)$  $\nu(\tilde{u},x_{0})=\nu(u,x_{0})$;

$(3)$ $\tilde{u}\in L^{\infty}_{loc}(V_{x_{0}}\setminus A)$, where {$A$ is the} analytic set $A:=\{z|\mathcal{I}(u)_{z}\neq\mathcal{O}_{z}\}$.
\end{Remark}

\subsection{Ohsawa-Takegoshi $L^2$ extension theorem}
$\\$

{We recall the statement of the Ohsawa-Takegoshi $L^2$ extension theorem.
It will allow us to argue by induction on dimension:}

\begin{Theorem}
\label{t:ot_plane}\cite{ohsawa-takegoshi}
Let $D$ be a bounded pseudo-convex domain in $\mathbb{C}^{n}$.
Let $u$ be a plurisubharmonic function on $D$.
Let $H$ be an $m$-dimensional complex plane in $\mathbb{C}^{n}$.
Then for any holomorphic function on $H\cap D$ satisying
$$\int_{H\cap D}|f|^{2}e^{-2u}d\lambda_{H}<+\infty,$$
there exists a holomorphic function $F$ on $D$ satisfying $F|_{H\cap D}=f$,
and
$$\int_{D}|F|^{2}e^{-2u}d\lambda_{n}\leq C_{D}\int_{H\cap D}|f|^{2}e^{-2u}d\lambda_{H},$$
where $C_{D}$ only depends on the diameter of $D$ and $m$,
and $d\lambda_{H}$ is the Lebesgue measure.
\end{Theorem}

In particular, there are two consequences of Theorem \ref{t:ot_plane} which will be used:

\begin{Remark}
\label{r:point}
Let $H$ be a point $z_{3}$ and $f=1$.
If $u(z_{3})<+\infty$,
then there exists a holomorphic function $F$ on $D$ satisfying $F|_{z_{3}}=1$,
and
$$\int_{D}|F|^{2}e^{-2u}d\lambda_{n}\leq C_{D} e^{-2u(z_{3})},$$
where $C_{D}$ only depends on the diameter of $D$.
\end{Remark}

Let $f\equiv 1$.
By Theorem \ref{t:ot_plane} and contradiction,
it follows that

\begin{Remark}
\label{r:integ_line}
Assume that $e^{-2u}$ is not integrable near $x_{0}$.
Then for any complex plane $H$ through $x_{0}$, $e^{-2u|_{H}}$ is not integrable near $x_{0}$.
\end{Remark}

\subsection{Lelong number along {a} complex line}
$\\$

We recall that
\begin{Lemma}
\label{l:lelong_fibre}(\cite{siu74}, see also \cite{demailly-book})
Let $u$ be a negative plurisubharmonic function on the unit ball $\mathbb{B}^{n}(x_{0},1)\subset\mathbb{C}^{n}$
satisfying $\nu(u,x_{0})=1$.
Then we have
$$\liminf_{r\to0}\frac{u(rz_{2}+x_{0})}{\log r}=1,$$
for almost all $z_{2}$ in the sense of the Lebesgue measure on sphere $S(x_{0},1)(=\partial\mathbb{B}^{n}(x_{0},1))$.
\end{Lemma}

It is known that (converse proposition of Lemma \ref{l:lelong_fibre}) if $\liminf_{r\to0}\frac{u(rz_{2}+x_{0})}{\log r}=1$ holds
for almost all $z_{2}$ in the sense of the Lebesgue measure on $S(x_{0},1)$,
then $\nu(u,x_{0})=1$ (\cite{siu74}, see also \cite{demailly-book}).

\begin{proof}(proof of Lemma \ref{l:lelong_fibre})
By {standard equivalent definitions of Lelong numbers} (see \cite{demailly2010}),
it follows that
\begin{equation}
\label{equ:20140603a}
\liminf_{r\to0}\frac{u(rz_{2}+x_{0})}{\log r}\geq\liminf_{z\to0}\frac{u(z+x_{0})}{\log |z|}=\nu(u,x_{0})=1.
\end{equation}
Recall a well-known equivalent form of the Lelong number of $u$ at $x_{0}$
(which can be chosen as definition of the Lelong number of $u$ at $x_{0}$):
$$\nu(u,x_{0})=\lim_{r\to0}\frac{\mu_{r}(u)}{\log r}$$
where $\mu_{r}(u)=\frac{\int_{S(x_{0},r)}udS}{Vol(S(x_{0},r))}$ (see \cite{demailly-book,demailly2010}).
Using $\nu(u,x_{0})=1$,
one can obtain
$$\lim_{r\to0}\frac{1}{Vol(S(x_{0},1))}\int_{S(x_{0},1)}\frac{u(rz_{2}+x_{0})}{\log r}dS(z_{2})=\lim_{r\to0}\frac{\mu_{r}(u)}{\log r}=1,$$
which implies
$$\frac{1}{Vol(S(x_{0},1))}\int_{S(x_{0},1)}\liminf_{r\to0}\frac{u(rz_{2}+x_{0})}{\log r}dS(z_{2})\leq 1,$$
(note that $\frac{u(rz_{2}+x_{0})}{\log r}>0$).
By equality \ref{equ:20140603a},
it follows that
$$\liminf_{r\to0}\frac{u(rz_{2}+x_{0})}{\log r}=1$$
for almost all $z_{2}\in S(x_{0},1)$ in the sense of the Lebesgue measure on $S(x_{0},1)$.
Then Lemma \ref{l:lelong_fibre} has been proved.
\end{proof}

Lemma \ref{l:lelong_fibre} is equivalent to the following
\begin{Remark}
\label{cor:line_Lelong}(\cite{siu74}, see also \cite{demailly-book})
Let $u$ be a plurisubharmonic function on the unit ball $\mathbb{B}^{n}(x_{0},1)\subset\mathbb{C}^{n}$.
Then $\nu(u,x_{0})=c$ if and only if
\begin{equation}
\label{equ:20140603e}
\nu(u|_{L},x_{0})=c,
\end{equation}
for almost all complex line through $x_{0}$ in the sense of the Lebesgue measure on $\mathbb{C}\mathbb{P}^{n-1}$.
\end{Remark}

{In fact it has been proved by Ben Messaoud and El Mir (see \cite{BenMir00})
that the set of $z_{2}$ for which the $p$ dimensional slice through $x_{0}$ has Lelong number larger than $\nu(u,x_{0})$ is pluripolar
in the corresponding Grassmannian manifold;
we will not need this result here. }

\subsection{Dimension of {varieties} along the fibres}\label{subsec:variety_fibre}
$\\$

Let $X:=\{z_{1}=\cdots=z_{n-1}=0\}$.
Consider a map $p$ from $\mathbb{C}^{n}\setminus X$ to $\mathbb{C}\mathbb{P}^{n-2}$:
$$p(z_{1},\cdots,z_{n})=(z_{1}:\cdots:z_{n-1}).$$

\begin{Lemma}
\label{l:dimen_fibre_zero}
Let $H$ be an analytic subvariety on $\mathbb{B}^{n}$ with dimension $n-k$,
such that $H\cap X=(0,\cdots,0)$.
Then the dimension of $H\cap p^{-1}(z_{1}:\cdots:z_{n-1})$ is $2-k$
for almost all $(z_{1}:\cdots:z_{n-1})$ in the sense of the Lebesgue measure on $\mathbb{C}\mathbb{P}^{n-2}$.
\end{Lemma}

\begin{proof}
We prove the present Lemma by contradiction:
if there exists a positive measure set $A\subset\mathbb{C}\mathbb{P}^{n-2}$,
such that $p^{-1}(z_{1}:\cdots:z_{n-1})\cap H$ is dimension $\geq 3-k$,
then the $2n-2k+2(=(2n-4)+(6-2k))$ Hausdorff measure of $H$ is not zero,
which contradicts to the complex dimension of $H$ is $n-k$.
\end{proof}

{It can be observed that the negligible set involved in the "almost all" conclusion is in
fact an algebraic subvariety directly related to the tangent cone of $H$ (which also has dimension $n-k$).}

\section{Proof of Theorem \ref{t:Lelong_lct}}

In the present section, we give a proof of Theorem \ref{t:Lelong_lct} and present some reformulations of the Theorem.

\subsection{Proof of Theorem \ref{t:Lelong_lct}}
$\\$

Without {loss} of generality, we assume that $x_{0}=\textbf{0}=(0,\cdots,0)\in\mathbb{C}^{n}$ and $u$ is negative.

We prove Theorem \ref{t:Lelong_lct} by contradiction and Theorem \ref{t:BM}:
if not
(if $(\{z|\nu(u,z)\geq 1\},x_{0})$ is dimension $(n-1)$ but not regular complex at $x_{0}$,
using Siu's decomposition theorem \cite{siu74}, one can obtain that $\nu(u,x_{0})\geq 2$
which contradicts to $\nu(u,x_{0})=1$),
then there exists a plurisubharmonic function $u$ on $\mathbb{B}^{n}(\textbf{0},1)$ satisfying that:

$(1)$ $e^{-2u}$ is not integrable near \textbf{0};

$(2)$ $\nu(u,\textbf{0})=1$;

$(3)$ there exists an analytic subvariety $H$ with complex dimension $n-2$ near \textbf{0},
such that $H\supset\{z|\nu(u,z)\geq 1\}$.

By Remark \ref{rem:GZ1005}, there exists a plurisubharmonic function $\tilde{u}$ on a neighborhood $V_{3}$ of \textbf{0}
such that

$(1)$ $e^{-2\tilde{u}}$ is not integrable near \textbf{0};

$(2)$ $\nu(\tilde{u},\textbf{0})=1$;

$(3)$  $\tilde{u}\in L^{\infty}_{loc}(V_{3}\setminus H)$.

Using Lemma \ref{cor:line_Lelong},
one can choose $X$ (as in subsection \ref{subsec:variety_fibre}) satisfying $\nu(\tilde{u}|_{X},\textbf{0})=1$.
By Lemma \ref{l:dimen_fibre_zero} $(k=2)$,
it follows that there exists $(z_{1}:\cdots:z_{n-1})$,
such that germ
$((p^{-1}(z_{1}:\cdots:z_{n-1})\cup X)\cap H,\textbf{0})=(\textbf{0},\textbf{0})$,
i.e. $(p^{-1}(z_{1}:\cdots:z_{n-1})\cup X)\cap H$ is isolated on the complex linear surfaces $p^{-1}(z_{1}:\cdots:z_{n-1})\cup X$.
Then it is clear that
\begin{equation}
\label{equ:20141005d}
\tilde{u}|_{p^{-1}(z_{1}:\cdots:z_{n-1})\cup X}\in L^{\infty}_{loc}((p^{-1}(z_{1}:\cdots:z_{n-1})\cup X)\cap V_{4}\setminus\textbf{0}),
\end{equation}
where $V_{4}$ is a small enough neighborhood of \textbf{0}.

Note that
$$1=\nu(\tilde{u},\textbf{0})\leq\nu(\tilde{u}|_{p^{-1}(z_{1}:\cdots:z_{n-1})\cup X},\textbf{0})\leq\nu(\tilde{u}|_{X},\textbf{0})=1,$$
then it is clear that
\begin{equation}
\label{equ:20141005c}
\nu(\tilde{u}|_{p^{-1}(z_{1}:\cdots:z_{n-1})\cup X},\textbf{0})=1.
\end{equation}

Using Remark \ref{r:integ_line} and $e^{-2\tilde{u}}$ is not integrable near \textbf{0},
one can obtain that
$$e^{-2\tilde{u}|_{p^{-1}(z_{1}:\cdots:z_{n-1})\cup X}}$$
is not integrable near \textbf{0}.
Combining equalities \ref{equ:20140603e} and \ref{equ:20141005c},
one can obtain that $\tilde{u}|_{p^{-1}(z_{1}:\cdots:z_{n-1})\cup X}$ satisfying:

$(1)$ $e^{-2\tilde{u}|_{p^{-1}(z_{1}:\cdots:z_{n-1})\cup X}}$ is not integrable near \textbf{0};

$(2)$ $\nu(\tilde{u}|_{p^{-1}(z_{1}:\cdots:z_{n-1})\cup X},\textbf{0})=1$;

$(3)$ $\tilde{u}|_{p^{-1}(z_{1}:\cdots:z_{n-1})\cup X}\in L^{\infty}_{loc}((p^{-1}(z_{1}:\cdots:z_{n-1})\cup X)\cap V_{4}\setminus\textbf{0})$.

Note that the existence of $\tilde{u}|_{p^{-1}(z_{1}:\cdots:z_{n-1})\cup X}$ contradicts Theorem \ref{t:BM} (
letting $\Omega=(p^{-1}(z_{1}:\cdots:z_{n-1})\cup X)\cap V_{4}$).
{Theorem \ref{t:Lelong_lct} is proved.}

\subsection{Reformulations of Theorem \ref{t:Lelong_lct}}
$\\$

By Siu's decomposition theorem \cite{siu74},
it follows that if $(\{z|\nu(u,z)\geq 1\},x_{0})$ is a germ of regular complex hypersurface,
then $e^{-2u}$ is not integrable near $x_{0}\in\Omega$.
Theorem \ref{t:Lelong_lct} can be reformulated:

\emph{Let $u$ be a plurisubharmonic function on $\Omega\subset\mathbb{C}^{n}$ satisfying
$\nu(u,x_{0})=1$.
Then $e^{-2u}$ is not integrable near $x_{0}\in\Omega$ if and only if
$(\{z|\nu(u,z)\geq 1\},x_{0})$ is a germ of regular complex hypersurface.}

There is an equivalent characterization of "$(\{z|\nu(u,z)\geq 1\},x_{0})$ is a germ of regular complex hypersurface":

The positive closed current $dd^{c}u$ near $x_{0}$
is the sum of the current of integration on a regular complex hypersurface ($\{z|\nu(u,z)\geq 1\}$)
and a current with zero Lelong number at $x_{0}$ (by Siu's decomposition theorem \cite{siu74}).

Then there is another reformulation of Theorem \ref{t:Lelong_lct}:

\emph{Let $u$ be a plurisubharmonic function on $\Omega\subset\mathbb{C}^{n}$ satisfying
$\nu(u,x_{0})=1$.
Then $e^{-2u}$ is not integrable near $x_{0}\in\Omega$ if and only if
$dd^{c}u$ near $x_{0}$
is the sum of the current of integration on a regular complex hypersurface through $x_{0}$
and a current with zero Lelong number at $x_{0}$.}

When $n=2$, the above statement {has been obtained in} \cite{FM05j}.

\section{Another proof of Theorem \ref{t:Lelong_lct}}

In the present section,
we present another proof of Theorem \ref{t:Lelong_lct} by using
Theorem \ref{t:Lelong_lct_dim2} instead of using Theorem \ref{t:BM}:

\subsection{H\"{o}lder inequality}
$\\$

Using {the} H\"{o}lder inequality and the openness conjecture, one can obtain the following Lemma:
\begin{Lemma}
\label{l:Holder}
Let $u=a_{1}u_{1}+u_{2}$ $(a_{1}>0)$ be {plurisubharmominc} on $\mathbb{B}^{n}(x_{0},1)$ satisfying that

$(1)$ $e^{-2u_{1}}$ is integrable near $x_{0}$;

$(2)$ $\nu(u_{2},x_{0})=a_{2}$;

$(3)$ $a_{1}+a_{2}=1$.

Then we have $e^{-2u}$ is  integrable near $x_{0}$.
\end{Lemma}

\begin{proof}
By the openness conjecture posed by Demailly-Kollar \cite{D-K01} which was recently proved by Berndtsson \cite{berndtsson13},
it follows that there exists $c>1$ such that $e^{-2cu_{1}}$ is integrable near $x_{0}$.
Using H\"{o}lder inequality, we obtain that
\begin{equation}
\label{equ:20140715a}
\int_{U}e^{-2u}\leq(\int_{U}e^{-2cu_{1}})^{\frac{a_{1}}{c}}(\int_{U}e^{-2\frac{1}{1-\frac{a_{1}}{c}}u_{2}})^{1-\frac{a_{1}}{c}}.
\end{equation}

By Theorem \ref{t:skoda}, it follows that
$\int_{U}e^{-2\frac{1}{1-\frac{a_{1}}{c}}u_{2}}$ is integrable for $U$ small enough.
Using inequality \ref{equ:20140715a},
one {obtains} the present Lemma.
\end{proof}

\subsection{Potentials of positive closed currents}
$\\$

In this subsection,
we recall some well-known results (see \cite{demailly-book}).
\begin{Lemma}
\label{l:composition}
Let $T$ be a positive closed current satisfying $T=T_{1}+T_{2}$,
where $T_1$ and $T_2$ are positive closed current on $\mathbb{B}^{n}$.
Then there are three plurisubharmonic functions $u$, $u_{1}$ and $u_{2}$,
satisfying
$T=dd^{c}u$, $T_{1}=dd^{c}u_{1}$ and $T_{2}=dd^{c}u_{2}$,
such that $u=u_{1}+u_{2}+v$ almost everywhere with respect to the Lebesgue measure on $\mathbb{B}^{n}$,
where $v$ is a pluriharmonic function on $\mathbb{B}^{n}$.
Moveover,
$u=u_{1}+u_{2}+v$ everywhere on $\mathbb{B}^{n}$.
\end{Lemma}

\begin{proof}
As $\mathbb{B}^{n}$ is a simply connected pseudoconvex domain in $\mathbb{C}^{n}$,
then there are three plurisubharmonic functions $u$, $u_{1}$ and $u_{2}$,
satisfying
$T=dd^{c}u$, $T_{1}=dd^{c}u_{1}$ and $T_{2}=dd^{c}u_{2}$.

Note that $dd^{c}(u-u_{1}-u_{2})=0$,
then $u-u_{1}-u_{2}$ is pluriharmonic in the sense of {distributions},
i.e.
there exists a pluriharmonic function $v$ such that $v=u-u_{1}-u_{2}$ almost everywhere in the sense of the Lebesgue measure on $\mathbb{B}^{n}$.
Then it follows that
$u=u_{1}+u_{2}+v$ almost everywhere in the sense of the Lebesgue measure on $\mathbb{B}^{n}$.
Considering the convolution,
it follows that
$$u\star\rho_{\varepsilon}=u_{1}\star\rho_{\varepsilon}+u_{2}\star\rho_{\varepsilon}+v\star\rho_{\varepsilon},$$
everywhere on $\mathbb{B}^{n}((0,\cdots,0),1-\varepsilon)$.
Note that $u=\lim_{\varepsilon\to0}u\star\rho_{\varepsilon}$,
$u_{1}=\lim_{\varepsilon\to0}u_{1}\star\rho_{\varepsilon}$,
$u_{2}=\lim_{\varepsilon\to0}u_{2}\star\rho_{\varepsilon}$,
$v=\lim_{\varepsilon\to0}v\star\rho_{\varepsilon}$.
Then $u=u_{1}+u_{2}+v$ everywhere on $\mathbb{B}^{n}$.
\end{proof}

By Siu's decomposition, we have
$$dd^{c}u=\sum_{j}\lambda_{j}[H_{j}]+\sum_{j'}\lambda_{j'}[H_{j'}]+S$$
where $\lambda_{j}>0$, $H_{j}$ is the analytic set through $x_{0}$ and $H_{j'}$ is the analytic set not through $x_{0}$,
and $S$ is the current satisfying that dimension of $\{\nu(u,z)\geq c\}$ is {at most equal to} $n-2$.
\begin{Corollary}
\label{cor:siu_comp}
There exist plurisubharmonic functions $u_{A}$, $u_{A'}$ and $u_{0}$,
satisfying
$$\sum_{j}\lambda_{j}[H_{j}]=dd^{c}u_{A},$$
$$\sum_{j'}\lambda_{j'}[H_{j'}]=dd^{c}u_{A'}$$
and
$$S=dd^{c}u_{0},$$
such that $u=u_{A}+u_{A'}+u_{0}+v$ everywhere on $\mathbb{B}^{n}$,
where $u_{A'}$ satisfying $\nu(u_{A'},x_{0})=0$.
\end{Corollary}

\begin{Remark}
\label{r:siu_comp}
As $H_{j'}$ is the analytic set not through $x_{0}$,
then the dimension of $(\{\nu(u_{A'},z)\geq c>0\},x_{0})$ is not bigger than $n-2$ for any $c$.
\end{Remark}

\subsection{Residual part in Siu's decomposition (Situation A.1)}

\begin{Lemma}
\label{l:dim_countable_fibre}
Let $\{Y_{j}\}$ be a {countable family} of analytic sets on $\mathbb{B}^{n}(x_{0},1)$ satisfying $dim Y_{j}\leq n-2$.
Let $H_{j'}$ be a {countable family} of irreducible analytic sets on $\mathbb{B}^{n}(x_{0},1)$ satisfying $dim H_{j'}=n-1$.
Then for almost all $(z_{1}:\cdots:z_{n-1})$ in the sense of the Lebesgue measure on $\mathbb{C}\mathbb{P}^{n-2}$,
we have

$(1)$ dimension of $Y_{j}\cap p^{-1}(z_{1}:\cdots:z_{n})$ is zero;

$(2)$ dimension of $H_{j}\cap p^{-1}(z_{1}:\cdots:z_{n})$ is not bigger than $1$.
\end{Lemma}

\begin{proof}
By Lemma \ref{l:dimen_fibre_zero}, one can obtain the present lemma.
\end{proof}

Using Lemma \ref{l:dim_countable_fibre},
one can obtain the following corollary:

\begin{Corollary}
\label{cor:dimen_fibre_zero_n-2}
Let $dd^{c}u:=S$ be the positive closed current on $\mathbb{B}^{n}$ satisfying that the dimension of $(\{z|\nu(u,z)\geq c\},x_{0})$
are smaller than $n-1$ for any $c>0$.
Then the dimension of $\{z|\nu(u|_{p^{-1}(z_{1}:\cdots:z_{n-1})\cup X},z)\geq c\}$ is zero
for almost all $(z_{1}:\cdots:z_{n-1})$ in the sense of the Lebesgue measure on $\mathbb{C}\mathbb{P}^{n-2}$.
Moreover,
for almost all $(z_{1}:\cdots:z_{n-1})$,
the dimension $\{z|\nu(u|_{p^{-1}(z_{1}:\cdots:z_{n-1})\cup X},z)\geq c\}$ is zero for any $c>0$.
\end{Corollary}

\begin{proof}
Using Lemma \ref{l:dim_countable_fibre},
one can obtain that
for almost all $(z_{1}:\cdots:z_{n-1})$,
the complex dimension of analytic set $p^{-1}(z_{1}:\cdots:z_{n-1}\cup X)\cap\{z|\nu(u,z)\geq c>0\}$ at $x_{0}$
is zero for any given $c>0$.
\end{proof}

\subsection{hypersurfaces not through $x_{0}$ (Situation A.2.1)}\label{sec:level_set}
$\\$

In the present section, we give some properties of plurisubharmonic function $u$ on $\Delta^{n}$ with zero Lelong number on any point in $\Delta$.

Let $p_{n}:\Delta^{n}\to\Delta^{n-2}$, where $p_{n}(z_{1},\cdots,z_{n})=(z'):=(z_{1},\cdots,z_{n-2})$.
\begin{Lemma}
\label{l:zero_lelong}
For almost all $z'\in\Delta^{n-2}$ (in the sense of the Lebesgue measure on $\Delta^{n-2}$),
the level set of Lelong numbers of $u|_{p_{n}^{-1}(z')}$
satisfies
$$\{z|\nu(u|_{p_{n}^{-1}(z')},z)=0\}=(\{z|\nu(u,z)=0\}\cap p_{n}^{-1}(z')).$$
\end{Lemma}

\begin{Remark}
\label{rem:zero_lelong}
It is known that $\nu(u|_{H},z)\geq \nu(u,z)$ for any regular complex variety $H\subset\Delta^{n}$ and $z\in H$,
which implies $\{z|\nu(u|_{p_{n}^{-1}(z')},z)=0\}\subseteq(\{z|\nu(u,z)=0\}\cap p_{n}^{-1}(z'))$ for any $z'\in\Delta^{n-2}$.
Then it suffices to prove
$$\{z|\nu(u|_{p_{n}^{-1}(z')},z)=0\}\supseteq(\{z|\nu(u,z)=0\}\cap p_{n}^{-1}(z')).$$
\end{Remark}

\begin{proof}(proof of Proposition \ref{l:zero_lelong})
By Theorem \ref{t:skoda},
it follows that $e^{-2ku}$ is locally integrable on the open subset $\{z|\nu(u,z)<\frac{1}{k}\}\subset\Delta^{n}$.

By Fubini's Theorem,
it follows that
the Lebesgue measure of
$$B_{k}:=\{z'|\int_{p_{n}^{-1}(z')\cap K_{m}}e^{-2ku|_{p_{n}^{-1}(z')}}=+\infty\}$$
is zero on $\Delta^{n-2}$
for any compact subset $K_{m}\subset\{z|\nu(u,z)<\frac{1}{k}\}$ satisfying $\cup_{m=1}^{\infty}K_{m}^{\circ}=\{z|\nu(u,z)<\frac{1}{k}\}$.
$$\int_{p_{n}^{-1}(z')\cap K_{m}}e^{-2ku|_{p_{n}^{-1}(z')}}<+\infty$$
(for any $m$) shows that
\begin{equation}
\label{equ:20140711a}
\nu(u|_{p_{n}^{-1}(z')},z)<\frac{2}{k}
\end{equation}
for any $z'\in(\Delta^{n-2}\setminus\cup_{l=1}^{k}B_{l})$ and
$z\in (p_{n}^{-1}(\Delta^{n-2}\setminus\cup_{l=1}^{k}B_{l})\cap\{z|\nu(u,z)<\frac{1}{k}\})$
(using the Skoda's result: $\nu(u|_{p_{n}^{-1}(z')},z)\geq 2$ implies that $e^{-2u|_{p_{n}^{-1}(z')}}$ is not integrable near $z\in p_{n}^{-1}(z')$).

Using inequality \ref{equ:20140711a},
we obtain
\begin{equation}
\label{equ:20140711b}
\nu(u|_{p_{n}^{-1}(z')},z)<\frac{2}{k}
\end{equation}
for any $k=1,2,\cdots$,
where $z'\in(\Delta^{n-2}\setminus\cup_{l=1}^{\infty}B_{l})$ and
$z\in (p_{n}^{-1}(\Delta^{n-2}\setminus\cup_{k=1}^{\infty}B_{k})\cap\{z|\nu(u,z)=0\})$.
Letting $k$ go to infinity,
using inequality \ref{equ:20140711b},
we obtain
$$\nu(u|_{p_{n}^{-1}(z')},z)=0$$
where $z'\in(\Delta^{n-2}\setminus\cup_{l=1}^{\infty}B_{l})$ and
$z\in (p_{n}^{-1}(\Delta^{n-2}\setminus\cup_{k=1}^{\infty}B_{k})\cap\{z|\nu(u,z)=0\})$,
i.e.
$$\{z|\nu(u|_{p_{n}^{-1}(z')},z)=0\}\supseteq(\{z|\nu(u,z)=0\}\cap p_{n}^{-1}(z')),$$
for any $z'\in (\Delta^{n-2}\setminus\cup_{k=1}^{\infty}B_{k})$.
Note that the Lebesgue measure of $\cup_{k=1}^{\infty}B_{k}$ is zero on $\Delta^{n-2}$.
Then we obtain the present Lemma.
\end{proof}

\begin{Lemma}
\label{l:dim_countable}
Let $\{Y_{j}\}$ be a {countable family} of analytic sets on $\mathbb{B}^{n}(x_{0},1)$ satisfying $dim Y_{j}\leq m-2$.
Let $H_{j'}$ be a {countable family} of irreducible analytic sets on $\mathbb{B}^{n}(x_{0},1)$ satisfying $dim H_{j'}=m-1$.
Let $H$ be an $(m-1)$-dimensional irreducible analytic set on $\mathbb{B}^{n}(x_{0},1)$ satisfying $H\subset(\cup_{j}Y_{j})\cup(\cup_{j'}H_{j'})$.
Then
there exists a $j'$ such that $H=H_{j'}$.
\end{Lemma}

\begin{proof}
Lemma \ref{l:dim_countable} can be proved by contradiction:
if {the conclusion does not hold}, then the $2n-3$ dimensional Hausdorff measure of $H_{j'}\cap H$ is zero for any $j'$.
It follows that
the $2m-3$ dimensional Hausdorff measure of $H\setminus(\cup_{j'}H_{j'})$ is positive.
Note that $(H\setminus(\cup_{j'}H_{j'}))\subset\cup_{j}Y_{j}$,
which contradicts to the fact that $2m-3$ dimensional Hausdorff measure of $Y_{j}$ is zero for any $j$,
then Lemma \ref{l:dim_countable} has been proved.
\end{proof}

Using Lemma \ref{l:dim_countable_fibre} and Lemma \ref{l:dim_countable} and combining with Lemma \ref{l:zero_lelong},
one can obtain the following corollary:

\begin{Corollary}
\label{cor:dimen_fibre_zero_n-1}
Let $dd^{c}u:=\sum_{j'}\lambda_{j'}[H_{j'}]$ be {a} positive closed current on $\mathbb{B}^{n}$,
where $H_{j'}'s$ are the hypersurfaces {that do not pass through} $x_{0}$.
Then for any given $c>0$ the dimension of $\{z|\nu(u|_{p^{-1}(z_{1}:\cdots:z_{n-1})},z)\geq c\}$ is zero
for almost all $(z_{1}:\cdots:z_{n-1})$ in the sense of the Lebesgue measure on $\mathbb{C}\mathbb{P}^{n-2}$.
Moreover,
for almost all $(z_{1}:\cdots:z_{n-1})$,
the dimension $\{z|\nu(u|_{p^{-1}(z_{1}:\cdots:z_{n-1})},z)\geq c\}$ is zero for any $c>0$.
\end{Corollary}

\begin{proof}
By Lemma \ref{l:zero_lelong},
it follows that for almost all $(z_{1}:\cdots:z_{n-1})$ in the sense of the Lebesgue measure on $\mathbb{C}\mathbb{P}^{n-2}$,
$$\{z|\nu(u|_{p^{-1}(z_{1}:\cdots:z_{n-1})},z)=0\}=(\{z|\nu(u,z)=0\}\cap p^{-1}(z_{1}:\cdots:z_{n-1}))$$
holds,
which implies
\begin{equation}
\label{equ:20140702b}
\begin{split}
\{z|\nu(u|_{p^{-1}(z_{1}:\cdots:z_{n-1})},z)\geq c\}\subseteq
&(p^{-1}(z_{1}:\cdots:z_{n-1})\setminus\{z|\nu(u,z)=0\})
\\=&
p^{-1}(z_{1}:\cdots:z_{n-1})\cap(\cup_{c>0}\{z|\nu(u,z)\geq c\})
\end{split}
\end{equation}
for any $c>0$.
Using Lemma \ref{l:dim_countable_fibre},
one can obtain that
for almost all $(z_{1}:\cdots:z_{n-1})$,
the dimension of germ $(p^{-1}(z_{1}:\cdots:z_{n-1})\cap\{z|\nu(u,z)\geq c>0\},x_{0})$
is zero for any $c>0$.
By Lemma \ref{l:dim_countable} ($m=2$, $H=\{z|\nu(u|_{p^{-1}(z_{1}:\cdots:z_{n-1})},z)\geq c\}$, by contradiction) and inequality \ref{equ:20140702b},
it follows that the dimension of $\{z|\nu(u|_{p^{-1}(z_{1}:\cdots:z_{n-1})\cup X},z)\geq c\}$ at $x_{0}$
is zero (for any $c>0$)
for almost all $(z_{1}:\cdots:z_{n-1})$ in the sense of the Lebesgue measure on $\mathbb{C}\mathbb{P}^{n-2}$.
\end{proof}

\subsection{Difference of analytic subvarieties along the fibres (Situation A.2.2.1)}
$\\$

Let $X:=\{z_{1}=\cdots=z_{n-1}=0\}$.
Consider a map $p$ from $\mathbb{C}^{n}\setminus X$ to $\mathbb{C}\mathbb{P}^{n-2}$:
$$p(z_{1},\cdots,z_{n})=(z_{1}:\cdots:z_{n-1}).$$

By contradiction, it is not hard to obtain the following Lemma:

\begin{Lemma}
\label{l:diff_fibre}
Let $H_{1}$ and $H_{2}$ be two different hypersurfaces on $\mathbb{B}^{n}$ satisfying
$H_{1}\cap X=H_{2}\cap X=(0,\cdots,0)\in\mathbb{C}^{n}$.
Then there exists  $(z_{1}:\cdots:z_{n-1})$ in $\mathbb{C}\mathbb{P}^{n-2}$,
such that
$H_{1}\cap p^{-1}(z_{1}:\cdots:z_{n-1})$ and $H_{2}\cap p^{-1}(z_{1}:\cdots:z_{n-1})$ are different.
\end{Lemma}

\subsection{Singularity of analytic subvariety along the fibres (Situation A.2.2.2)}
$\\$

\begin{Lemma}
\label{l:regular}
Let $p_{n}$ be the same mapping as in subsection \ref{sec:level_set}.
Let $H_{0}=\{h_{0}=0\}$ ($h$ is the defining function) be a regular complex hypersurface on $\Delta^{n}$ through $x$.
Assume that $p_{n}(x)$ is not the critical value of $p_{n}|_{H_{0}}: H_{0}\to \Delta^{n-2}$.
Then
$dh|_{p_{n}^{-1}(p_{n}(x))}$ is not varnishing at $x$.
\end{Lemma}

\begin{proof}
We prove the present Lemma by contradiction:
if $dh|_{p_{n}^{-1}(p_{n}(x))}$ is {not} varnishing at $x$,
then the linear space of $T_{x}\{h=0\}\subset T_{x}\Delta^{n}$ contains $T_{x}p_{n}^{-1}(p_{n}(x))\subset T_{x}\Delta^{n}$,
i.e.
$(p_{n}|_{x})_{\star}T_{x}H_{0}$ is degenerate,
which contradicts {the fact that} $p_{n}(x)$ is not the critical value of $p_{n}|_{H_{0}}$.
\end{proof}

Let $X:=\{z_{1}=\cdots=z_{n-1}=0\}$.
Consider a map $p$ from $\mathbb{C}^{n}\setminus X$ to $\mathbb{C}\mathbb{P}^{n-2}$:
$$p(z_{1},\cdots,z_{n})=(z_{1}:\cdots:z_{n-1}).$$

Let $H=\{h=0\}$ be a hypersurface on $\mathbb{B}^{n}$,
which is singular at $(0,\cdots,0)$,
where $h$ is the defining function of $H$, such the restriction $dh$ of any component of the set of regular points in $H$ is not varnishing identically.
Assume that $X\cap H=(0,\cdots,0)$.

\begin{Lemma}
\label{l:sing_fibre}If there exists a point $(z_{1}:\cdots:z_{n-1})$ satisfying that $p^{-1}(z_{1}:\cdots:z_{n-1})\cap H_{reg}$ ($H_{reg}:=\{z|dh|_{z}\neq0\}$) is dimension $1$,
then for a positive measure of $(z_{1}:\cdots:z_{n-1})$
in the sense of the Lebesgue measure on $\mathbb{C}\mathbb{P}^{n-2}$, $h_{z'}(:=h|_{p^{-1}(z_{1}:\cdots:z_{n-1})\cup X})$ satisfies:

$(1)$ $h_{z'}$ is singular at $\textbf{0}:=(0,\cdots,0)\in\mathbb{C}^{n}$;

$(2)$ $dh$ is not varnishing identically on any open subset of $p^{-1}(z_{1}:\cdots:z_{n-1})\cap H$.
\end{Lemma}

\begin{Remark}
\label{r:sing_fibre_06}
It is clear that the complex dimension of
$\{z|\nu(dd^{2}\log|h_{z'}|,z)\geq\nu(dd^{2}\log|h_{z'}|),\textbf{0}\}$
is zero for a positive measure of $(z_{1}:\cdots:z_{n-1})$
in the sense of the Lebesgue measure on $\mathbb{C}\mathbb{P}^{n-2}$.
\end{Remark}

\begin{proof}(proof of Lemma \ref{l:sing_fibre})
Note that $dh|_{(0,\cdots,0)}$ is varnishing,
then it suffices to prove that $dh|_{p^{-1}(z_{1}:\cdots:z_{n-1})\cap H}$ is not varnishing identically on any open subset of $p^{-1}(z_{1}:\cdots:z_{n-1})\cap H$.

As there exists a point $(z_{1}:\cdots:z_{n-1})$ satisfying that $p^{-1}(z_{1}:\cdots:z_{n-1})\cap H_{reg}$ is dimension $1$,
then there exists an open subset $U$ of $\mathbb{C}\mathbb{P}^{n-2}$ contained in the image $p(H_{reg}\setminus X)$.
By Sard's Theorem,
it follows that Lebesgue measure of the critical value set $A_{crit}$ of $p|_{H_{reg}\setminus X}$ is zero on $\mathbb{C}\mathbb{P}^{n-2}$.
By Lemma \ref{l:regular},
it follows that
$p^{-1}(z_{1}:\cdots:z_{n-1})\cap H_{reg}$ is smooth for any point $(z_{1}:\cdots:z_{n-1})\in(U\setminus A_{crit})$.
Note that $U\setminus A_{crit}$ has positive measure  on $\mathbb{C}\mathbb{P}^{n-2}$,
then we prove the present Lemma.
\end{proof}

\begin{Remark}
\label{r:sing_fibre}Assume that for all point $(z_{1}:\cdots:z_{n-1})$ satisfying that $p^{-1}(z_{1}:\cdots:z_{n-1})\cap H_{reg}\neq\emptyset$,
the complex dimension of $p^{-1}(z_{1}:\cdots:z_{n-1})\cap H$ is $2$, i.e. $(p^{-1}(z_{1}:\cdots:z_{n-1})\cup X)\subset H$.
Then the dimension of $p^{-1}(z_{1}:\cdots:z_{n-1})\cap H$ is $0$ for almost all $(z_{1}:\cdots:z_{n-1})$
in the sense of the Lebesgue measure on $\mathbb{C}\mathbb{P}^{n-2}$.
\end{Remark}

$\\$
\textbf{Proof of Theorem \ref{t:Lelong_lct}:}
$\\$

Without {loss} of generality, we assume that $x_{0}=\textbf{0}=(0,\cdots,0)\in\mathbb{C}^{n}$ and $u$ is negative.

Using Siu's decomposition,
we have
$$dd^{c}u=\sum_{j}\lambda_{j}[H_{j}]+\sum_{j'}\lambda_{j'}[H_{j'}]+S,$$
where $\lambda_{j},\lambda_{j'}>0$, $H_{j}$ is through $\textbf{0}$, and $H_{j'}$ is not through $\textbf{0}$
(the dimension of $\{\nu(u_{0}+u_{A'},z)\geq c\}$ at $\textbf{0}$ is small than $n-1$ for any $c>0$).

By Lemma \ref{l:composition},
it follows that there exist plurisubharmonic functions $u_{A}$, $u_{A'}$ and $u_{0}$,
such that

$(1)$ $u=u_{A}+u_{A'}+u_{0}$;

$(2)$ $dd^{c}u_{A}=\sum_{j}\lambda_{j}[H_{j}]$;

$(3)$ $dd^{c}u_{A'}=\sum_{j'}\lambda_{j'}[H_{j'}]$;

$(4)$ $dd^{c}u_{0}=S$.
$\\$

If $\{\nu(u,z)\geq 1\}$ is not a regular complex hypersurface through $\textbf{0}$, then there are two situations:

$(A.1)$ $\nu(u_{0},\textbf{0})>0$;

$(A.2.)$ $\nu(u_{0},\textbf{0})=0$ ($\Leftrightarrow\nu(u_{A'}+u_{A},\textbf{0})=1$).

By Lemma \ref{l:lelong_fibre},
it follows that one can assume that $\nu(u|_{X},\textbf{0})=1$.
Then we obtain that
$$1\leq\nu(u|_{X},\textbf{0})\leq\nu(u|_{p^{-1}(z_{1}:\cdots:z_{n-1})\cup X},\textbf{0})\leq\nu(u,\textbf{0})=1,$$
which shows that
$$\nu(u|_{p^{-1}(z_{1}:\cdots:z_{n-1})\cup X},\textbf{0})=1$$
for any $(z_{1}:\cdots:z_{n-1})\in\mathbb{C}\mathbb{P}^{n-2}$.

\subsection{Situation (A.1)}
$\\$

By Lemma \ref{l:Holder} ($u_{1}=\frac{u_{0}}{\nu(u_{0},\textbf{0})}$, $u_{2}=u_{A}+u_{A'}$, $a_{1}=\nu(u_{0},\textbf{0})$),
it suffices to prove that $u=u_{0}$ ($\nu(u_{0},\textbf{0})=1$).
Using Corollary \ref{cor:dimen_fibre_zero_n-2},
there exists a complex plane $V=(p^{-1}(z_{1}:\cdots:z_{n})\cup X)$ (dimension 2) through $\textbf{0}$,
such that dimension of $\{\nu(u|_{V},z)\geq c\}$ at $\textbf{0}$ is zero for any $c>0$.

By Theorem \ref{t:Lelong_lct_dim2} and Remark \ref{r:integ_line},
it follows that Situation (A.1) {is solved}.

\subsection{Situation (A.2)}
$\\$

By Lemma \ref{l:Holder} ($u_{1}=\frac{u_{A}+u_{A'}}{\nu(u_{A}+u_{A'},\textbf{0})}$, $u_{2}=u_{0}$, $a_{1}=\nu(u_{A}+u_{A'},\textbf{0})$),
it suffices to consider $u_{0}=0$:

$(A.2.1)$ $\nu(u_{A'},\textbf{0})>0$;

$(A.2.2.)$ $\nu(u_{A'},\textbf{0})=0$ ($\Leftrightarrow\nu(u_{A},\textbf{0})=1$)

\subsubsection{Situation (A.2.1)}

By Lemma \ref{l:Holder} ($u_{1}=\frac{u_{A'}}{\nu(u_{A'},\textbf{0})}$, $u_{2}=u_{A}$, $a_{1}=\nu(u_{A'},\textbf{0})$),
it suffices to prove that $u=u_{A'}$ ($\nu(u_{A'},\textbf{0})=1$).
Using Corollary \ref{cor:dimen_fibre_zero_n-1},
there exists a complex plane $V=(p^{-1}(z_{1}:\cdots:z_{n})\cup X)$ (dimension 2) through $\textbf{0}$,
such that dimension of $\{\nu(u|_{V},z)\geq c\}$ at $\textbf{0}$ is zero for any $c>0$.

By Theorem \ref{t:Lelong_lct_dim2} and Remark \ref{r:integ_line},
it follows that Situation (A.2.1) {is solved}.

\subsubsection{Situation (A.2.2)}
$\\$

By Lemma \ref{l:Holder} ($u_{1}=\frac{u_{A}}{\nu(u_{A},\textbf{0})}$, $u_{2}=u_{A'}$, $a_{1}=\nu(u_{A},\textbf{0})$),
it suffices to prove that $u=u_{A}$:

$(A.2.2.1)$ there exists $j_{1}$ and $j_{2}$, such that $H_{j_1}$ and $H_{j_2}$ are different hypersurfaces;

$(A.2.2.2)$ there exists only one $H_{j}$ which is singular at $\textbf{0}$, such that $\nu(\lambda_{j}[H_{j}],\textbf{0})=1$.

\subsubsection{Situation (A.2.2.1)}
$\\$

By Lemma \ref{l:Holder}($dd^{c}u_{1}=
\frac{\lambda_{1}[H_{1}]+\lambda_{2}[H_{2}]}{\nu(\lambda_{1}[H_{1}]+\lambda_{2}[H_{2}],\textbf{0})}$, $a_{1}=\nu(\lambda_{1}[H_{1}]+\lambda_{2}[H_{2}],\textbf{0})$),
it suffices to prove that $dd^{c}u=\lambda_{1}[H_{1}]+\lambda_{2}[H_{2}]$.

By Lemma \ref{l:lelong_fibre},
it follows that one can assume that $\nu(u|_{X},x_{0})=1$,
and $H_{1}\cap X=H_{2}\cap X=x_{0}$ by choosing neighborhood of $x_{0}$ small enough.

Using Lemma \ref{l:diff_fibre},
we can obtain a point $(z_{1},\cdots,z_{n})\in\mathbb{C}\mathbb{P}^{n-2}$,
such that $(p^{-1}(z_{1},\cdots,z_{n})\cup X)\cap H_{j_{i}}$ $(i=1,2.)$ are different curves.
Using Theorem \ref{t:Lelong_lct_dim2} and Remark \ref{r:integ_line},
then it follows that Situation (A.2.2.1) {is solved}.

\subsubsection{Situation (A.2.2.2)}
$\\$

By Lemma \ref{l:Holder}($dd^{c}u_{1}=
\frac{\lambda_{0}[H_{0}]}{\nu(\lambda_{0}[H_{0}],\textbf{0})}$,
$a_{1}=\nu(\lambda_{0}[H_{0}],\textbf{0})$),
it suffices to prove that $dd^{c}u=\lambda_{0}[H_{0}]$, which is singular at $\textbf{0}$.

By Lemma \ref{l:lelong_fibre},
it follows that one can assume that $\nu(u|_{X},x_{0})=1$,
and $H_{0}\cap X=x_{0}$ by choosing neighborhood of $x_{0}$ small enough.

Using Lemma \ref{l:sing_fibre} and Remark \ref{r:sing_fibre},
we can obtain a point $(z_{1},\cdots,z_{n})\in\mathbb{C}\mathbb{P}^{n-2}$,
such that $(p^{-1}(z_{1},\cdots,z_{n})\cup X)\cap H_{j}$ is a singular curve,
or discrete points.
Using Theorem \ref{t:Lelong_lct_dim2} and Remark \ref{r:integ_line},
then it follows that Situation (A.2.2.2) {is solved}.
$\\$

Combining with Situations (A.1) and (A.2), Theorem \ref{t:Lelong_lct} {is solved}.

\section{A new proof of Theorem \ref{t:skoda}}

In this section, we give a new proof of Theorem \ref{t:skoda} by using the Ohsawa-Takegoshi $L^{2}$  extension theorem
{in a "dynamical manner"}.

\subsection{Estimation of $L^2$ norm of holomorphic functions on unit disc}
$\\$

We recall a Lemma which was used
in \cite{GZopen-a,GZopen-b} to prove the strong openness conjecture:

\begin{Lemma}
\label{l:open_b}(see \cite{GZopen-a,GZopen-b})
Let $f_{a}$ be a holomorphic function on unit disc $\Delta\subset\mathbb{C}$,
which satisfies
$f|_{o}=0$ and $f_{a}(a)=1$ for any $a$,
then we have
$$\int_{\Delta_{r}}|f_{a}|^{2}d\lambda_{1}>C_{1}|a|^{-2},$$
where $a\in\Delta$ whose norm is smaller than $\frac{1}{6}$,
$C_{1}$ is a positive constant independent of $a$ and $f_{a}$.
\end{Lemma}

\subsection{Using {the} $L^2$ extension theorem {dynamically} along {the} radius}
$\\$

In this subsection,
we use {the} $L^2$ extension theorem {dynamically} along {the} radius and obtain the following result:

\begin{Proposition}
\label{p:movable_radius}
 Let $u$ be a plurisubharmonic function on the unit ball $B(x_{0},1)\subset\mathbb{C}^{n}$
satisfying that $e^{-2u}$ is not integrable near $x_{0}$.
Then for any $z_{2}=(z^{1}_{2},z^{2}_{2},\cdots,z^{n}_{2})$ satisfying $|z_{2}|=1$,
we have
\begin{equation}
\label{equ:20140603b}
u(rz_{2}+x_{0})< C_{3}+\log r
\end{equation}
holds for any $r<\frac{1}{4}$,
where $C_{3}$ is a constant independent of $r$ and $z_{2}$.
If $u$ is negative,
then $C_{3}$ is independent of $u$.
\end{Proposition}

\begin{proof}
It suffices to consider the case that $u<0$.

Let $H$ be a complex line through $x_{0}$ and $z_{2}+x_{0}$,
then $D:=H\cap B(x_{0},1)$ is a unit disc.
Using Remark \ref{r:point},
we obtain a holomorphic function $F_{rz_{2}}$ on $D$ satisfying
$F_{rz_{2}}(rz_{2})=1$ and
\begin{equation}
\label{equ:20140512a}
\int_{D}|F_{rz_{2}}|^{2}e^{-2u}d\lambda_{1}\leq C_{D} e^{-2u(rz_{2}+x_{0})}.
\end{equation}

As $u<0$,
it follows that
\begin{equation}
\label{equ:20140512b}
\int_{D}|F_{rz_{2}}|^{2}d\lambda_{1}<\int_{D}|F_{rz_{2}}|^{2}e^{-2u}d\lambda_{1}.
\end{equation}

Using Remark \ref{r:integ_line}, we obtain $F_{rz_{2}}|_{x_{0}}=0$.
By Lemma \ref{l:open_b} $(a=rz_{2},$ and $f_{a}=F_{rz_{2}}|_{D})$,
it follows that
\begin{equation}
\label{equ:20140512c}
C_{1}|r|^{-2}<\int_{D}|F|^{2}d\lambda_{1}.
\end{equation}

By inequalities \ref{equ:20140512a}, \ref{equ:20140512b} and \ref{equ:20140512c},
it follows that
\begin{equation}
\label{equ:20140512d}
C_{1}|r|^{-2}<C_{D} e^{-2u(rz_{2}+x_{0})},
\end{equation}
i.e.
\begin{equation}
\label{equ:20140512e}
u(rz_{2}+x_{0})<\log r+\frac{\log C_{D}-\log C_{1}}{2}.
\end{equation}
Note that $C_{D}$ and $C_{1}$ are both uniform constant,
then proposition \ref{p:movable_radius} {is proved}.
\end{proof}

By inequality \ref{equ:20140603b}, it follows that
\begin{equation}
\label{equ:20140603c}
\liminf_{r\to 0}\frac{u(rz_{2}+x_{0})}{\log r}\geq 1
\end{equation}
for any $z_{2}$, which implies
\begin{equation}
\label{equ:20140603d}
\nu(u|_{L},x_{0})\geq 1,
\end{equation}
for any complex line through $x_{0}$.
Then Theorem \ref{t:skoda} follows from Proposition \ref{p:movable_radius} (by contradiction).

%

\bibliographystyle{references}
\bibliography{xbib}

\end{document}